\documentclass[fleqn]{article}
\usepackage{amsmath,amssymb,amsfonts,amsthm,pdfpages,hyperref,mathtools,esint}
\usepackage[english]{babel}
\usepackage{framed}
\usepackage{authblk}
\usepackage[T1]{fontenc}
\usepackage[utf8]{inputenc}

\DeclareMathAlphabet{\mathonebb}{U}{bbold}{m}{n}

\newcommand{\Z}{{\mathbb Z}}
\newcommand{\R}{{\mathbb R}}
\newcommand{\RR}{{\mathbb{R}}}
\newcommand{\CC}{{\mathbb{C}}}
\newcommand{\norm}[1]{\left\Vert{#1}\right\Vert}
\newcommand{\abs}[1]{\left\vert{#1}\right\vert}
\newcommand{\nnn}{\nonumber}

\title{Vortex structure in $p$-wave superconductors}
\author[1]{Stan Alama \thanks{{\ttfamily alama@mcmaster.ca}}}
\author[1]{Lia Bronsard \thanks{{\ttfamily bronsard@mcmaster.ca}}}
\author[1,2]{Xavier Lamy \thanks{{\ttfamily xlamy@math.univ-lyon1.fr}}}
\affil[1]{\footnotesize{Department of Mathematics and Statistics, McMaster University, Hamilton, ON, L8S 4K1, Canada.}}
\affil[2]{\footnotesize{Universit\'e de Lyon, Institut Camille Jordan,  CNRS UMR 5208, Universit\'e Lyon 1,  France.}}

\newtheorem{thm}{Theorem}[section]
\newtheorem{prop}[thm]{Proposition}
\newtheorem{lem}[thm]{Lemma}

\theoremstyle{definition}

\DeclareMathOperator{\curl}{curl}

\begin{document}
\date{\today}
\maketitle
\begin{abstract}
We study vortices in p-wave superconductors in a Ginzburg-Landau setting.  The state of the superconductor is described by a pair of complex wave functions, and the p-wave symmetric energy functional couples these in both the kinetic (gradient) and potential energy terms, giving rise to systems of partial differential equations which are nonlinear and coupled in their second derivative terms.  We prove the existence of energy minimizing solutions in bounded domains $\Omega\subset\RR^2$, and consider the existence and qualitative properties (such as the asymptotic behavior) of equivariant solutions defined in all of $\RR^2$.  
The coupling of the equations at highest order changes the nature of the solutions, and many of the usual properties of classical Ginzburg-Landau vortices either do not hold for the p-wave solutions or are not immediately evident.
\end{abstract}

\section{Introduction}

With the discovery of high temperature superconductors physicists have investigated many new and unusual families of superconducting materials, many with properties which are quite different from the metal superconductors which were originally studied a century ago.  Among these is Sr${}_2\,$Ru\,O${}_4$, which (although it is not a high temperature superconductor) has a layered perovskite crystalline structure which is very similar to the cuprate high $T_C$ materials.  This material is special, however, in that it has a different electronic structure from conventional ``s-wave'' superconductors described by the microscopic BCS model, but instead exhibits a ``p-wave'' electron pairing symmetry (see \cite{agter98}).  Superconductors with p-wave pairing develop such unconventional properties as spontaneous magnetization and surface currents\cite{heebagterberg99, kallin}, and square vortex lattices in certain parameter regimes \cite{agter98}.

In this paper we consider a Ginzburg--Landau model for p-wave superconductors in two dimensions.  The state of the superconductor is described by a pair of complex wave functions, $\eta=(\eta_-,\eta_+): \ \Omega\subset\RR^2\to \CC^2$ and the magnetic vector potential, $A: \ \Omega\to\RR^2$.  The p-wave symmetry is encoded in the kinetic energy by means of an anisotropic gradient term,
\begin{equation*}
E(\eta,A) = \int \left( e_{kin}(\eta,A) + {\kappa^2} e_{pot}(\eta) + |\curl A|^2 \right),
\end{equation*}
where
\begin{align*}
e_{kin}(\eta,A) & = |D\eta_+ |^2 +|D\eta_- |^2  + (1+\nu)\left[ D_x\eta_+ \cdot D_x\eta_- -D_y\eta_+ \cdot D_y \eta_-\right] \\
&\quad + (1-\nu)\left[ D_x\eta_-\wedge D_y\eta_+ - D_x\eta_+\wedge D_y\eta_-\right]  
\end{align*}
and
\begin{equation}\label{potential}
e_{pot}(\eta) = \frac 12 (|\eta_+ |^2-1)^2 + \frac 12 (|\eta_- |^2-1)^2 + 2 |\eta_+ |^2 |\eta_- |^2 + \nu (\eta_+^2)\cdot (\eta_-^2).
\end{equation}
Here $\kappa$ is the Ginzburg--Landau parameter, $\nu \in (-1,1)$ is an anisotropy parameter, 
and the operator $D=\nabla -i A$.  The dot and wedge product on $\CC$ are calculated by treating $z=x+iy\in\CC$ as a real vector $(x,y)\in\RR^2$, and applying the usual definitions.

By writing the potential energy in the form,
\begin{align*}
e_{pot} & = \frac 12 + \frac 12 (|\eta_+|^2+ |\eta_-|^2 -1)^2 + (1-|\nu|)|\eta_+|^2|\eta_-|^2 \\
&\quad + |\nu | \left[|\eta_+|^2|\eta_-|^2 +\mathrm{sign}(\nu) (\eta_+^2) \cdot (\eta_-^2)\right],
\end{align*}
we note that for $-1 <\nu <1$, the minimum of the potential $e_{pot}$ is attained exactly at
\begin{equation*}
(\eta_-,\eta_+) = (1,0) \text{ or } (0,1).
\end{equation*}
Thus, we expect that energy minimizers will have this form away from any vortices, with one ''dominant" component, which we take to be $\eta_-$, $|\eta_-|\simeq 1$, and one ``admixed'' component \cite{heebagterberg99} $\eta_+$ which is small in the bulk of the sample.

Also note that  $E$ is gauge invariant: for smooth enough $\varphi$, 
\begin{equation*}
E(\eta_\pm, A) = E(e^{i\varphi}\eta_\pm, A+\nabla\varphi).
\end{equation*}

The goal of this paper is to study isolated vortices in this p-wave Ginzburg--Landau model, and thus we concentrate on energy minimizing solutions with given degrees imposed on the boundary of a disk or at infinity, in the case of entire solutions (defined on $\Omega=\RR^2$.)  As in the classical Ginzburg--Landau functional, in questions concerning isolated vortices the role of the magnetic field $h=\curl A$ is secondary, and so we neglect the vector potential $A$ in this paper.  We expect that our results should extend to the full system with vector potential with some minor technical adjustments.  With this simplification, the energy functional takes the form:
$$  E(\eta)  = \int_\Omega \left[ e_{kin}(\eta) + \kappa^2 e_{pot}(\eta)\right] dx, $$
with $e_{pot}$ as before, and 
\begin{align}
e_{kin}(\eta) & = |\nabla\eta_+ |^2 +|\nabla\eta_- |^2  + (1+\nu)\left[ \partial_x\eta_+ \cdot \partial_x\eta_- -\partial_y\eta_+ \cdot \partial_y \eta_-\right] \nnn\\
\nnn
&\qquad + (1-\nu)\left[ \partial_x\eta_-\wedge \partial_y\eta_+ - \partial_x\eta_+\wedge \partial_y\eta_-\right]  
\\  \label{kinetic}
& = |\nabla\eta_+ |^2 +|\nabla\eta_- |^2  + (\Pi_- \eta_+)\cdot (\Pi_+\eta_-) + \nu (\Pi_+\eta_+)\cdot (\Pi_-\eta_-),
\end{align}
with operators $\Pi_+ = \Pi = \partial_x + i \partial_y,$ $\Pi_- = -\Pi^* = \partial_x-i \partial_y$.  
As we will see shortly, the kinetic energy is nonnegative, but not coercive: it vanishes along a nontrivial linear subspace of functions $\eta$.  This is an early indication of the difficulties involved in the analysis of the p-wave functional.  Energy minimizers solve a system of Euler--Lagrange equations, which are coupled in the second derivative terms:
\begin{equation}\label{GLsyst}
\left.
\begin{aligned}
2\Delta \eta_- + [\Pi_-^2 + \nu\Pi_+^2]\eta_+ & = \kappa^2\left( 2\eta_-(|\eta_-|^2 -1)
     + 4\eta_-|\eta_+|^2 + 2\nu\overline{\eta}_-\eta_+^2\right) \\
2\Delta \eta_+ + [\Pi_+^2 + \nu\Pi_-^2]\eta_- & = \kappa^2\left( 2\eta_+(|\eta_+|^2 -1)
     + 4\eta_+|\eta_-|^2 + 2\nu\overline{\eta}_+ \eta_-^2\right) \\
\end{aligned}
\right\}
\end{equation}

Our first result concerns the existence of energy minimizing solutions in any smooth bounded simply connected domain $\Omega\subset\RR^2$.  Consider the Dirichlet boundary condition 
\begin{equation}\label{DirBC}
\eta_\pm|_{\partial\Omega} =  g_\pm,
\end{equation}
where  $g_\pm: \ \partial\Omega\to  \CC$ are given smooth functions.

\begin{thm}\label{thm1}  Let $g_\pm\in H^{1/2}(\partial\Omega)$ and define
\begin{equation*}
W=\{\eta\in H^1(\Omega;\CC^2): \ \text{\eqref{DirBC} is satisfied}\}.
\end{equation*}
Assume that $(c_++\alpha z, c_- -\alpha \bar{z})\not\in W$ for any constants $\alpha, c_\pm\in\CC$.
Then, there exist a minimizer of $E(\eta)$ in $W$.

In particular, there exists a minimizer in $\Omega=B_R$ for $g_\pm= \alpha_\pm e^{in_\pm\theta}$ provided that one of $n_\pm\neq \pm 1$ or $\alpha_+\neq -\alpha_-$.
\end{thm}
We recall that the potential energy is minimized with $|\eta_-|=1$, $|\eta_+|=0$ (or vice-versa,) and hence a natural choice of boundary condition is
\begin{equation}\label{symbc}
\eta_- |_{\partial \Omega}= e^{i n\theta}, \qquad \eta_+ |_{\partial\Omega} =0,
\end{equation}
with $n\in\mathbb{N}$, in analogy with Ginzburg--Landau vortices but recognizing the bulk states preferred by $e_{pot}$. 
Theorem~\ref{thm1} is proved in section \ref{s_exist}.  There we show that the restriction on the boundary data can compensate for the general lack of coercivity in the whole space $H^1(\Omega)$.
 
As in the classical Ginzburg--Landau model, it is to be expected that the symmetric (equivariant) vortex solutions, $\eta_\pm=f(r)e^{i n_\pm\theta}$, play a special role.  Here we already see the effect of the p-wave symmetry, as radial solutions do not exist in general, but only for certain choices of the parameters.  Indeed, in section \ref{s_entire} we show that equivariant solutions cannot exist for anisotropy $\nu\neq 0$, and that for $\nu=0$ there is a restriction on the degrees, $n_+=n_-+2$. 

Assuming $\nu=0$ and $n_+=n_- +2$, the equivariant ansatz reduces the problem to finding real-valued functions $(f_-(r),f_+(r))$, $r\in (0,\infty)$, which solve the Euler-Lagrange equations, a system of two coupled second-order ordinary differential equations (see \eqref{systf1} below.)  As with the classical Ginzburg--Landau model, entire solutions (in all $\RR^2$) with nontrivial degree at infinity must have infinite energy.  We thus adopt the strategy of passing to the limit in balls $B_R$ of increasing radius, in which we minimize the energy subject to the boundary condition \eqref{symbc} on $\partial\Omega=\partial B_R$.  Even in this simpler context, there are significant obstacles to overcome.  Although the existence of solutions in the balls $B_R$ is guaranteed by Theorem~\ref{thm1}, for general $n\in\mathbb{N}$ the coupling of the system at highest order prevents us from obtaining the necessary {\em a priori} estimates to pass to the limit $R\to\infty$, except when $n=-1$.  For $n=-1$, which is the most physically relevant case \cite{heebagterberg99}, we prove:

\begin{thm}\label{thm2}
There exists a smooth entire equivariant solution 
$\eta=(\eta_-,\eta_+)=(f_-(r)e^{-i\theta},f_+(r)e^{+i\theta})$ to the Ginzburg--Landau system \eqref{GLsyst}, with $f_-(r)\to 1$ and $f_+(r)\to 0$ as $r\to\infty$. Moreover it holds
\begin{equation}\label{asymptot1}
f_-= 1 - \frac{1}{2 r^{2}} -\frac{7}{4 r^{4}} + O(r^{-6}),\qquad
f_+ = -{1\over 2 r^2} -{13\over 4r^4} + O(r^{-6}),
\end{equation}
as $r\to +\infty$.
\end{thm}
The existence of entire equivariant solutions with degrees $(n, n+2)$, $n\neq -1$, is an open problem, as is uniqueness.

Given the usual interpretation of $f_\pm(r)$ as a local density of superconducting electrons, we would expect that these solutions have fixed sign.  This is a nontrivial question, as the coupling of the two components in the kinetic energy term precludes the usual arguments used in Ginzburg--Landau vortices, and even the methods developed for semilinear Ginzburg--Landau systems \cite{ABM1} fail in this context.  To obtain a result in this direction we introduce an additional parameter into the model, and employ perturbative methods.  For $t\in [0,1]$, we consider the family of functionals,
\begin{equation}\label{E_t}
E_t (\eta;R)= \int_{B_R} (|\nabla \eta_+|^2 +|\nabla \eta_-|^2 + t(\Pi_+ \eta_-) \cdot (\Pi_- \eta_+) + e_{pot}).
\end{equation}
When $t=0$ the system couples only through the potential energy term.  Vortices in a two-component model with similar potential energy were studied by Lin \& Lin \cite{lin2003vortices}, and with an applied magnetic field by Alama \& Bronsard \cite{ABp1,ABp2}. 
With the equivariant ansatz $\eta=(\eta_-,\eta_+)=(f_-(r)e^{-i\theta},f_+(r)e^{+i\theta})$, the Euler--Lagrange equations take the form
\begin{equation}\label{GLt}
\begin{gathered}
\Delta_r f_- -\frac{1}{r^2}f_- + \frac t2 (\Delta_r f_+ -\frac{1}{r^2}f_+ ) = f_- (f_-^2-1)+2f_-f_+^2,\\
\Delta_r f_+ -\frac{1}{r^2}f_+ + \frac t2 (\Delta_r f_- -\frac{1}{r^2}f_- ) = f_+ (f_+^2-1)+2f_+f_-^2.
\end{gathered}
\end{equation}
When $t=1$, this is exactly the system satisfied by the physical p-wave functions with the equivariant ansatz and $n=-1$.  On the other hand, when $t=0$ the system \eqref{GLt} partially decouples, and admits a solution of the form $f^0=(f^0_-, f^0_+)=(f,0)$, with $f(r)$ the radial degree-one Ginzburg--Landau vortex profile.  We verify that $f^0$ gives a nondegenerate locally minimizing solution to the system \eqref{GLt} at $t=0$, and the solutions for $t>0$ are obtained via the Implicit Function Theorem.
In section \ref{s_perturb} we prove:
\begin{thm}\label{thm3}
There exists $t_0$ such that for all $t\in (0,t_0)$ there exist smooth bounded solutions $(f^t_-, f^t_+)$ of \eqref{GLt} such that:
\begin{enumerate}
\item[(a)] $f^t_-(0)=0=f^t_+(0)$;
\item[(b)] $f^t_-(r)\to 1$, $f^t_+(r)\to 0$ as $r\to\infty$;
\item[(c)] $0<f_-^t(r)<1$, $f^t_+(r)<0$ for all $r\in (0,\infty)$;
\item[(d)] As $r\to\infty$,
\begin{equation*}
f_-^t= 1 - \frac{1}{2 r^{2}} -\frac{5t^2+9}{8 r^{4}} + O(r^{-6}), \qquad
f_+^t = t\left[-{1\over 2 r^2} -{13\over 4r^4} + O(r^{-6})\right].
\end{equation*}
\end{enumerate}
\end{thm}
Note that $0>f_+(r)=-|\eta_+|$, and so the components of the equivariant solution incorporate a relative phase shift of $\pi$, in addition to having conjugate phases.
The asymptotic estimate in (d) may be made uniform for $r\ge R$ and $t\in (0,t_0)$; see Theorem~\ref{asymptotics} for a more precise statement.
We note that it is thanks to the uniform bounds on the asymptotic error that we may obtain the global control of the signs of the components in (c).  Our result does not preclude the possiblity that one or both of $f^t_\pm$ vanishes or changes sign at some value of $t\in (0,1]$.  If this were to occur at some $t$, the solution $\eta_\pm=f^t_\pm e^{\pm i \theta}$ would still be a valid solution to the system of equations, but with a very unconventional profile for vortices.  We conjecture that in fact (c) remains valid for all $t\in (0,1]$, but again this question is open.

The methods employed in this paper extend various techniques used to study vortices in Ginzburg--Landau systems.  In particular, the perturbation arguments rely on the extensive analysis of the linearization of the classical Ginzburg--Landau functional by Mironescu \cite{mironescu95}.  The asymptotic expansion follows the basic strategy followed in \cite{alamagao13}, based on \cite{CEQ}.  The use of perturbative methods to study entire vortex solutions to the d-wave symmetric coupled Ginzburg--Landau system were also introduced by Kim \& Phillips \cite{kimphillips} and Han \& Lin \cite{hanlin}, although their approach was different from ours.

\medskip

\noindent
{\bf Acknowledgements:} \ The authors are grateful to C. Kallin and E. Taylor of the McMaster Physics department for helpful discussions concerning the p-wave model.  SA and LB are supported by an NSERC (Canada) Discovery Grant. XL is supported by a Labex Milyon doctoral mobility scholarship for his stay at McMaster University. He wishes to thank McMaster's Department of Mathematics and Statistics for their hospitality, and his Ph.D. advisor P.~Mironescu for his constant support.

\section{Existence of minimizers}\label{s_exist}

We begin with the existence of minimizers for the general functional 
\begin{equation*}
E(\eta)=\int_{\Omega} \left( e_{kin}(\eta) + \kappa^2 e_{pot}(\eta) \right)\, dx
\end{equation*}
with $e_{kin}$ as in \eqref{kinetic}, $e_{pot}$ as in \eqref{potential}, and with Dirichlet boundary condition \eqref{DirBC}.
The existence of minimizers, even in a bounded domain $\Omega\subset\mathbb R^2$, is not obvious, since the kinetic energy is not coercive:

\begin{prop}\label{zerokin}
For any given $\eta_{\pm}\in H^1(\Omega)$, it holds that $e_{kin}(\eta)\geq 0$, with equality if and only if
\begin{equation*}
\eta_+ = c_+ +\alpha z,\quad \eta_- = c_- -\alpha \overline z,
\end{equation*}
for some $c_+,c_-,\alpha\in\mathbb C$.
\end{prop}
\begin{proof}
The kinetic energy may be rewritten as
\begin{align*}
e_{kin} 
& = \frac{1+\nu}{2}|\partial_x\eta_+ + \partial_x \eta_- |^2 + \frac{1+\nu}{2} |\partial_y\eta_+ - \partial_y\eta_-|^2 \\
&\quad +\frac{1-\nu}{2}|\partial_y\eta_+ + i \partial_x\eta_-|^2 + \frac{1-\nu}{2}|\partial_x\eta_+ + i \partial_y\eta_- |^2.
\end{align*}
In particular it is non-negative, and  $e_{kin}=0$ implies
\begin{equation*}
\partial_x[\eta_+ +\eta_-]=0,\quad\partial_y[\eta_+-\eta_-]=0,\quad\text{and }(\partial_x +i\partial_y)\eta_+ =0.
\end{equation*}
Thus there exist one-dimensional distributions $u,v\in\mathcal D'(\mathbb R)$ such that
\begin{equation*}
\eta_+ =u (y) + v(x),\quad \eta_- =u(y) -v(x),\quad iu'(y)+v'(x) =0.
\end{equation*}
Differentiating the last equation, we deduce that $u''=v''=0$. Therefore $u$ and $v$ are affine functions with $u'=iv'$:
\begin{equation*}
u=u_0 +i \alpha y,\quad v =v_0 +\alpha x,\quad\text{ for some }\alpha\in\mathbb C,
\end{equation*}
and we obtain the desired conclusion with $c_+ =u_0 + v_0$ and $c_- = u_0-v_0$.
\end{proof}

As a consequence of Proposition~\ref{zerokin}, there is no hope for a general inequality of the form $\int e_{kin} \geq c \| \nabla \eta \|_{L^2}^2$ to be valid. However, we have the following:

\begin{lem}\label{weakcoerc}
Let $\Omega$ be an open subset of $\mathbb R^2$. Let $W\subset H^1(\Omega)^2$ be a closed affine subspace such that
\begin{equation*}
W\cap \left\lbrace (c_+ +\alpha z , c_- -\alpha \bar z)\colon c_\pm,\alpha\in\mathbb C \right\rbrace =\emptyset.
\end{equation*}
Then there exists $c>0$ (depending on $\Omega$ and $W$) such that
\begin{equation*}
\int_\Omega e_{kin}(\eta) \geq c \|\eta\|^2_{H^1}
\end{equation*}
for every $\eta\in W$.
\end{lem}
\begin{proof}
We argue by contradiction. If the conclusion does not hold, then (using the homogeneity of the involved quantities) there exists a sequence $(\eta^k)\subset W$ such that
\begin{equation*}
\|\eta^k\|_{H^1}=1,\quad \int e_{kin}(\eta^k) \longrightarrow 0.
\end{equation*}
Up to considering a subsequence, and since $W$ is weakly closed, we may assume that $\eta^k$ converges $H^1$-weakly to $\eta\in W$. On the other hand, since the kinetic energy is convex (as a non-negative quadratic form), it holds
\begin{equation*}
\int e_{kin}(\eta)\leq \liminf \int e_{kin}(\eta^k) = 0,
\end{equation*}
so that by Lemma~\ref{zerokin}, $\eta_\pm =c_\pm +\alpha (y \pm ix)$, thus contradicting the assumption on $W$.
\end{proof}

In particular, we may impose Dirichlet boundary conditions ensuring that the assumption of Lemma~\ref{weakcoerc} is satisfied. For instance, the following result will allow us to construct -- in Section~\ref{s_entire} below -- physically relevant `radial vortex' solutions.

\begin{proof}[Proof of Theorem~\ref{thm1}]
The first assertion follows from Proposition~\ref{zerokin} and Lemma~\ref{weakcoerc}.  In the case $\Omega=B_R$, $g_\pm=\alpha_\pm e^{in_\pm\theta}$, 
 it suffices to show that for any $c_\pm,\alpha\in \mathbb C$,
\begin{equation*}
\eta_\pm =c_\pm \pm \alpha re^{\pm i\theta}\notin W,
\end{equation*}
which follows from the uniqueness of Fourier decomposition on $\partial B_R$.
\end{proof}

\section{Entire vortex solutions}\label{s_entire}

In this section we study symmetric vortices, that is, solutions of the form
\begin{equation*} 
\eta_\pm (re^{i\theta})=f_\pm(r)e^{in_\pm \theta},\quad n_\pm\in\Z,
\end{equation*}
where $f_\pm$ are real-valued functions.
However, because of the coupling term in the kinetic energy, and in contrast with other coupled systems of Ginzburg-Landau equations \cite{alamagao13}, not all values of $n_\pm\in\Z$ are natural. 

Indeed, the existence of such symmetric solutions is related to invariance properties of the energy. More specifically, for any $n_\pm\in\Z$, one may define an action of $\mathbb S^1$ on functions $\eta_\pm(z)$:
\begin{equation*}
(\omega\cdot\eta_\pm)(z)=\omega^{n_\pm}\eta(\omega^{-1}z),\quad\omega\in\mathbb S^1.
\end{equation*}
A straightforward computation shows that
\begin{equation*}
\begin{split}
E(\eta)-E(\omega\cdot\eta)&=\int ( [1-\omega^{n_+-n_--2}]\Pi_-\eta_+)\cdot(\Pi_+\eta_-) \\
&\quad +\nu\int ( [1-\omega^{n_+-n_-+2}]\Pi_+\eta_+)\cdot(\Pi_-\eta_-) \\
&\quad+ \kappa^2\nu \int ([1-\omega^{2(n_+-n_-)}]\eta_+^2)\cdot(\eta_-^2 ).
\end{split}\end{equation*}
Hence we see that, in the case $\nu=0$, the energy is invariant if and only if 
\begin{equation*}
n_+=n_-+2.
\end{equation*}
In the case $\nu\neq 0$, the energy can not be invariant, and the only invariance that can be expected is for the subgroup $\mathbb U_4\subset\mathbb S^1$, which explains why vortices with square symmetry are predicted \cite{heebagterberg99,shiraishimaki99}.

In view of the above discussion, we consider from now on the case $\nu=0$. Moreover, since we will be interested in solutions defined in the whole plane $\R^2$, the parameter $\kappa$ can be scaled out, and we assume also $\kappa=1$.
In that case the Euler-Lagrange equations read
\begin{equation}\label{systeta}
\begin{gathered}
\Delta \eta_- +\frac 12 \Pi_-^2 \eta_+  
=  \eta_- (|\eta_- |^2-1) + 2\eta_- | \eta_+ |^2,\\
\Delta \eta_+  +\frac 12 \Pi_+^2 \eta_- 
= \eta_+ (|\eta_+ |^2-1) + 2\eta_+ | \eta_- |^2.
\end{gathered}
\end{equation}
in terms of $f_\pm$ defined by \eqref{fpm}, and using the notation $\Delta_r f = r^{-1}(rf')'=f''+r^{-1}f'$,
the system \eqref{systeta} takes the form,
\begin{equation}\label{systf1}
\begin{aligned}
\Delta_r f_- & -\frac{n^2}{r^2}f_- +\frac 12 \left(\Delta_r f_+ +2\frac{n+1}{r}f_+'+\frac{n(n+2)}{r^2}f_+ \right) \\
&\qquad =f_-(|f_-|^2-1)+2f_-f_+^2,\\
\Delta_r f_+ & -\frac{(n+2)^2}{r^2}f_+ +\frac 12 \left(\Delta_r f_- -2\frac{n+1}{r}f_-' + \frac{n(n+2)}{r^2}f_- \right) \\
&\qquad = f_+(|f_+|^2-1)+2f_+f_-^2.
\end{aligned}
\end{equation}

In the following we will show the existence of entire solutions of \eqref{systf1} with $n=-1$, that is equivariant solutions of the form
\begin{equation} \label{fpm}
\eta_-(re^{i\theta})=f_-(r)e^{- i\theta},\qquad \eta_+(re^{i\theta})=f_+(r)e^{+ i\theta},
\end{equation}
where $f_\pm$ are real-valued functions. This is the choice of degrees made in \cite{heebagterberg99}, in the expectation that these solutions are the ``most stable''.  In fact, the choice $n=-1$ simplifies the equations by eliminating a troublesome first order cross term in each equation.  Existence of entire equivariant solutions for $n\neq -1$ remains an open problem.

With the choice $n=-1$, the kinetic energy becomes
\begin{equation}\label{ekinf}
e_{kin} =|f'|^2 +\frac{1}{r^2}|f|^2 
+ \left(f_-'+\frac 1r f_- \right)\left(f_+'+\frac 1r f_+ \right),
\end{equation}
where $|f'|^2=(f_-')^2+(f_+')^2$ and $|f|^2=f_-^2+f_+^2$.
Moreover, the system \eqref{systeta} reads
\begin{equation}\label{systf}
\begin{aligned}
\Delta_r f_- -\frac{1}{r^2}f_- +\frac 12\left( \Delta_r f_+ -\frac{1}{r^2}f_+ \right) 
& =f_-(|f_-|^2-1)+2f_-f_+^2,\\
\Delta_r f_+  -\frac{1}{r^2}f_+ +\frac 12 \left(\Delta_r f_- - \frac{1}{r^2}f_- \right) 
& = f_+(|f_+|^2-1)+2f_+f_-^2.
\end{aligned}
\end{equation}

Note that the continuity of $\eta_\pm$ forces $f_\pm$ to satisfy homogeneous boundary conditions at the origin:
\begin{equation}\label{fpm0}
f_-(0)=f_+(0)=0.
\end{equation}
In fact these conditions \eqref{fpm0} are automatically satisfied by any bounded solutions of \eqref{systf}.
As for boundary conditions at $\infty$ we impose, in agreement with \eqref{symbc},
\begin{equation}\label{bcvort}
\lim_{r\to\infty} (f_-,f_+)=(1,0).
\end{equation}

The strategy to obtain entire solutions of \eqref{systf}-\eqref{bcvort} is standard: we first obtain solutions in balls $B_R$ by direct minimization, and then let $R\to\infty$. We denote by $\mathcal H_R$ the admissible energy space for vortex configurations in $B_R$:
\begin{equation}\label{H_R}
\begin{aligned}
\mathcal H_R & =\left\lbrace \text{ real-valued } (f_-,f_+) \colon \eta_\pm=f(r)e^{\pm i\theta}\in H^1(B_R) \right\rbrace \\
& = \left\lbrace \text{ real-valued } (f_-,f_+)\colon \int_0^R \left( |f'|^2 + \frac{1}{r^2}|f|^2\right)rdr <\infty\right\rbrace.
\end{aligned}
\end{equation}
We also denote by $\mathcal H^{bc}_R$ the vortex configurations in $B_R$, having the right boundary conditions at $R$, and by $H_R^0$ the admissible perturbations, i.e. with zero boundary conditions at $R$:
\begin{gather}
\mathcal H_R^{bc} = \left\lbrace (f_-,f_+)\in\mathcal H_R \colon f_-(R)=1,\: f_+(R)=0 \right\rbrace,\label{H_Rbc}\\
\mathcal H_R^{0} = \left\lbrace (\varphi_-,\varphi_+)\in\mathcal H_R \colon \varphi_-(R)=\varphi_+(R)=0 \right\rbrace. \label{H_R0}
\end{gather}

To obtain entire solutions of \eqref{systf}-\eqref{bcvort}, we will need two kinds of \textit{a priori} estimates on solutions in $\mathcal H_R$: an $L^\infty$ bound, and a bound on the potential energy.

\begin{lem}\label{lem_apriori}
Let $f_\pm\in\mathcal H^{bc}_R$ solve \eqref{systf} in $(0,R)$, with $f_-(R)=1$, $f_+(R)=0$.
 Then it holds
 \begin{equation}\label{potbound}
2\int_0^R e_{pot}\, rdr \leq 1.
\end{equation}
and
\begin{equation}\label{unifbound}
f_-^2+f_+^2\leq 3 \quad \text{in }(0,R).
\end{equation}
If in addition we know that $f_-\ge 0$ and $f_+\le 0$ in $(0,R)$, then we have 
\begin{equation}\label{unifbound2}
f_-^2+f_+^2\leq 1 \quad \text{in }(0,R).
\end{equation}
\end{lem}

\begin{proof}[Proof of the $L^\infty$ estimate \eqref{unifbound} and \eqref{unifbound2}:]
We use the weak formulation of the system \eqref{systf}. That is, for any test functions $\varphi_\pm\in \mathcal H_R^0$, it holds
\begin{equation}\label{systfweak}
\begin{split}
&\int_0^R \Bigg\lbrace f_-'\varphi_-' +\frac{1}{r^2}f_-\varphi_- + f_+'\varphi_+' +\frac{1}{r^2}f_+\varphi_+ \\
&\quad\quad +\frac{1}{2}(f_-' +\frac 1r f_-)(\varphi_+' +\frac 1r \varphi_+) + \frac 12 (f_+'+\frac 1r f_+)(\varphi_-'+\frac 1r \varphi_-)\Bigg\rbrace rdr\\
&\quad = -\frac 12 \int_0^R \left[ De_{pot}(f)\cdot\varphi \right] rdr,
\end{split}
\end{equation}
where
\begin{equation}\label{Depotf}
\frac 12 De_{pot}(f)\cdot\varphi = (2f_+^2 + f_-^2-1)f_-\varphi_- + (2f_-^2 + f_+^2-1)f_+\varphi_+.
\end{equation}
We apply that weak formulation to test functions of the form $\varphi_\pm =f_\pm V$, where $V\geq 0$ will be chosen appropriately later on. We find 
\begin{equation*}
\begin{split}
&\int_0^R\Bigg\lbrace
e_{kin}(f)V +\frac 12 (f_-^2+f_+^2 +f_+f_-)'V' +\frac 1r f_+f_-V'
\Bigg\rbrace rdr \\
&\quad = -\frac 12\int_0^R V \left( De_{pot}(f)\cdot f\right) rdr.
\end{split}
\end{equation*}
Integrating by parts, we rewrite that last equation as
\begin{equation}\label{weakV}
\begin{split}
&\int_0^R \Bigg\lbrace  \left(e_{kin}-\frac 1r(f_-f_+)'\right)V +\frac 12 (f_-^2+f_+^2 +f_+f_-)'V'\Bigg\rbrace rdr \\
&\quad = -\frac 12 \int_0^R V \left( De_{pot}(f)\cdot f\right) rdr.
\end{split}
\end{equation}
Next we notice that
\begin{equation}\label{Depotff}
\begin{split}
De_{pot}(f)\cdot f & 
= (f_-^2 + f_+^2 -1)(f_-^2+f_+^2)+2f_-^2f_+^2 \\
&\geq 0 \qquad\text{if }f_-^2 + f_+^2\geq 1,
\end{split}
\end{equation}
and
\begin{equation}\label{ekin++}
\begin{split}
e_{kin}-\frac 1r (f_-f_+)'& = (f_-')^2 +(f_+')^2 +f_-'f_+' + \frac{1}{r^2}\left( f_-^2+ f_+^2 +f_-f_+\right)\\ &\geq 0.
\end{split}
\end{equation}
Now we can choose the function $V$. We define, for an arbitrary $M>0$,
\begin{equation}\label{UV}
U=\max( f_-^2+f_+^2+f_+f_- -3/2, 0)\quad\text{and } V=\min (U,M).
\end{equation}
 It is easy to check that $\varphi_\pm = f_\pm V\in \mathcal H_R$ are indeed admissible test functions \eqref{H_R0}. Plugging \eqref{UV} into \eqref{weakV}, and using the inequalities \eqref{Depotff} and \eqref{ekin++}, we obtain
\begin{equation}\label{ineqV}
\int_0^R (V')^2 rdr \leq 0,
\end{equation}
and therefore $V=0$ a.e. We deduce that
\begin{equation*}
f_+^2+f_-^2+f_+f_-\leq 3/2,
\end{equation*}
which obviously implies the $L^\infty$ estimate \eqref{unifbound}. 

In case that $f_-\ge 0$ and $f_+\le 0$ in $(0,R)$, let $W=f_-^2+f_+^2-1$.  If $W$ attains a positive maximum at $r\in(0,R)$, we easily compute
\begin{align*}
0\ge \Delta_r W(r) &\geq 2 f_-\Delta_r f_- + 2 f_+\Delta_r f_+\\
& = 2W(W+1) + 4f_-^2f_+^2 -f_-f_+(3f_-^2+3f_+^2 -2)+\frac{2}{r^2}(f_-^2+f_+^2) \\
& \geq 2W(W+1)>0,
\end{align*}
thus proving \eqref{unifbound2}.
\end{proof}

\begin{proof}[Proof of the potential energy estimate \eqref{potbound}:]
The potential energy estimate is classically proven using a Pohozaev identiy. The Pohozaev identity is obtained by multiplying the first line of \eqref{systf} by $r^2f_-'$ and the second line by $r^2 f_+'$, and adding them. The resulting equality can be rewritten as
\begin{equation}\label{poho}
\begin{gathered}
\left[r^2(f_-')^2+r^2(f_+')^2+r^2f_+'f_-'-f_-^2-f_+^2 - f_-f_+
\right]'\\
\qquad =r^2 \left[e_{pot}\right]' = \left[r^2 e_{pot} \right]'-2r(e_{pot}).
\end{gathered}
\end{equation}
Integrating \eqref{poho} from 0 to $R$ and using the boundary conditions $f_\pm(0)=0$, $f_\pm(R)=(0,1)$, we obtain
\begin{equation}\label{potint}
2\int_0^R (e_{pot})\, rdr =1 -R^2\left[ f_-'(R)^2+f_+'(R)^2 + f_-(R)f_+(R)\right]\leq 1,
\end{equation}
thus proving \eqref{potbound}.
\end{proof}

With the \textit{a priori} estimates of Lemma~\ref{lem_apriori} at hand, we are ready to prove Theorem~\ref{thm2}.

\smallskip

\begin{proof}[Proof of Theorem~\ref{thm2}]
We prove here the existence part of Theorem~\ref{thm2}. The asymptotic expansion \eqref{asymptot1} is then a consequence of Theorem~\ref{asymptotics}, which is proven in Section~\ref{s_asympt}.

 We proceed in three steps: first we show the existence of solutions in finite balls, then let the radii tend to $+\infty$ and obtain entire solutions of \eqref{systf} , and eventually we show that those solutions satisfy the boundary conditions \eqref{bcvort}. The first two steps are fairly standard after the preliminary work in Section~\ref{s_exist} and the uniform bound of Lemma~\ref{lem_apriori}. The last step classically relies on the potential energy bound of Lemma~\ref{lem_apriori}, but requires an extra argument that was not needed in previous related works (as e.g. \cite{alamagao13}).

\textbf{Step 1:} Existence of solutions $f_\pm^R$ in $(0,R)$ with $f^R_\pm(R)=(0,1)$. 

By Lemma~\ref{weakcoerc}, the kinetic energy functional is coercive on the closed affine (real) subspace
\begin{equation}\label{spaceminR}
\left\lbrace \eta_\pm =f_\pm(r)e^{\pm i\theta} \colon f\in \mathcal H_R^{bc} \right\rbrace \subset H^1(B_R)^2.
\end{equation}
Therefore the direct method of the calculus of variation ensures the existence of a minimizer $\eta_\pm = f_\pm^R(r)e^{\pm i\theta}$. The functions $f_\pm$ solve \eqref{systf} in (0,R). Moreover, $f_\pm\in \mathcal H_R^{bc}$ and Lemma~\ref{lem_apriori} applies: it holds
\begin{equation*}
|f|^2 \leq 3,\quad 2\int_0^R e_{pot}(f) \, rdr \leq 1.
\end{equation*}
Note that the $L^\infty$ bound \eqref{unifbound} ensures that $\Delta_r f_\pm\in L^\infty_{loc}$, and therefore by elliptic regularity $f_\pm$ are smooth.

\textbf{Step 2:} Taking the limit as $R\to\infty$. 

We regard $f_\pm^R$ as being defined on $(0,\infty)$ by setting $f_\pm^R \equiv (0,1)$ in $(R,\infty)$. Thanks to the $L^\infty$ bound $|f|^2\leq 3$, elliptic estimates ensure that $(f_\pm^R)'$ is uniformly bounded in any compact interval of $(0,\infty)$. Hence we may extract a converging subsequence
\begin{equation*}
f_\pm^{R_n}\longrightarrow f_\pm \quad\text{ locally uniformly in }(0,R).
\end{equation*}
It follows that $f_\pm$ are smooth bounded solutions of \eqref{systf}.

\textbf{Step 3:} Boundary conditions \eqref{bcvort}. 

From the bound on the potential energy \eqref{potbound} and Fatou's lemma, we obtain that
\begin{equation*}
\int_0^\infty e_{pot} \, rdr <\infty.
\end{equation*}
We claim that this finite energy property implies that $\lim_{r\to\infty} e_{pot}=0$. To this end, remark that it holds $|f'_\pm(r)|\leq C(1+r)$, which is easily established using the uniform bound $|f_\pm|\leq 3$ together with the differential system \eqref{systf} satisfied by $f_\pm$. Now assume that there exists a subsequence $r_n\to\infty$ such that $e_{pot}(r_n)\geq \varepsilon>0$. We may assume in addition that $r_{n+1}-r_n\geq 1$. From $|f_\pm|\leq 3$ and $|f_\pm'|\leq C(1+r)$ we obtain that $|e_{pot}'|\leq C(1+r)$, and we deduce that there exists $\delta>0$ such that $e_{pot}\geq \varepsilon/2$ on $(r_n-\delta/r_n,r_n+\delta/r_n)$. But this would imply
\begin{equation*}
\int_0^\infty e_{pot}\, rdr \geq \frac\varepsilon 2\sum_n \int_{r_n-\delta/r_n}^{r^n+\delta/r_n} rdr =\frac\varepsilon 2\sum_n 2\delta =\infty,
\end{equation*}
which contradicts the finite energy property. Therefore it holds
\begin{equation*}
\lim_{r\to\infty} e_{pot} =0.
\end{equation*}
On the other hand, recall that $e_{pot}=0$ exactly at the points $(0,1)$ and $(1,0)$.
As a consequence, any converging subsequence $f_\pm(r_n)$ must converge to either $(0,1)$ or $(1,0)$. 

In fact only one of these two points can be such a limit: if there exists sequences $f_\pm(r^1_n)\to (0,1)$ and $f_\pm(r^2_n)\to (1,0)$, then using the continuity of $f_\pm$ one easily constructs a sequence $r^3_n\to\infty$ such that
\begin{equation*} \mathrm{dist}( f_\pm(r^3_n),\lbrace (0,1),(1,0)\rbrace)\geq 1/2.
\end{equation*}
But then one could extract a subsequence $f_\pm(r^3_{n'})\to \ell_\pm \notin \lbrace (0,1),(1,0)\rbrace$,  contradicting the fact that $\lim e_{pot}=0$. 

Therefore there is a unique possible limit for converging subsequences $f_\pm(r_n)$, and we conclude that the limit $f_\pm(\infty)$ exists and is either $(0,1)$ or $(1,0)$. Up to exchanging $f_+$ with $f_-$ (the equations are symmetric), we have the right boundary conditions at $\infty$.
\end{proof}

Now that we have the existence of entire vortices, we would like to investigate qualitative properties of the radial profiles $f_\pm$. The first natural question is whether or not they have a sign. In the classical one-component Ginzburg-Landau setting \cite{bbh}, as in other two-component models \cite{alamagao13}, existence of the radial profile components with a sign follow from a simple energy argument: replacing $f$ with $|f|$ or $-|f|$ does not increase the energy. In the present case however, this argument does not work, because of the coupling term in the kinetic energy.

If there do exist radial profiles with a sign, it is clear that $f_-$ should be positive since $f_-(\infty)=1$. On the other hand, due to the asymptotic expansion \eqref{asymptot1},  $f_+$ should be negative.
This is in agreement with numerical computations performed in \cite{shiraishimaki99}. In the next section we give arguments supporting the conjecture that $f_-\geq 0$ and $f_+\leq 0$. We consider a perturbed model and prove the existence of vortices with such signs.

\section{Vortex structure for a perturbed model}\label{s_perturb}

This section is devoted to proving Theorem~\ref{thm3}. We start by presenting and proving the main tools needed in the proof.

\subsection{Main ingredients}\label{ss_ingredients}

Recall that we consider the family of perturbed functionals \eqref{E_t} and we look for radial vortex solutions of the form
\begin{equation*}
\eta_+ = f_+(r)e^{i\theta},\quad \eta_- =f_-(r)e^{-i\theta}.
\end{equation*}
Then the energy \eqref{E_t} becomes
\begin{equation}\label{I_t}
I_t(f;R) := \int_0^R \left( |f'|^2 + \frac{1}{r^2}|f|^2 + t(f_-'+\frac 1r f_-)(f_+'+\frac 1r f_+) + e_{pot}\right)rdr,
\end{equation}
where $|f'|^2=(f_-')^2+(f_+')^2$ and $|f|^2=f_-^2+f_+^2$, and the corresponding Euler-Lagrange equations are \eqref{GLt}. 

 The solutions $f^t$ are obtained by perturbation of a solution $f^0$ of \eqref{GLt} for $t=0$, given by
\begin{equation*}
f_-^0 =f,\quad f_+^0 \equiv 0,
\end{equation*}
where $f$ is the classical Ginzburg-Landau radial vortex profile solving
\begin{equation}\label{fradialGL}
\Delta_r f - \frac{1}{r^2}f = f(f^2-1),\qquad f(0)=0,\quad f(\infty)=1.
\end{equation}
More specifically, the solution $f^t$ will be of the form
\begin{equation*}
f^t=f^0 + g^t,\qquad g_\pm^t(\infty)=0.
\end{equation*}
Perturbed solutions will be obtained through the implicit function theorem, and to this end we need a stability result. The space of admissible perturbation is
\begin{equation}
\mathcal H = \left\lbrace \varphi_\pm\in H^1_{loc}(0,\infty)\colon \eta=\varphi_\pm(r)e^{\pm i\theta}\in H^1(\mathbb R^2)\right\rbrace.
\end{equation}
Although the entire solution $f^0$ does not have finite energy $I_0$ in $(0,\infty)$, it makes sense to consider variations with respect to compact perturbations: for $\varphi_\pm\in C_c^\infty (0,\infty)$, such that $\mathrm{supp}\, \varphi_\pm\subset (0,R_0)$, it holds
\begin{equation*}
I_0(f^0+\varphi; R_0)-I_0(f^0;R_0)=
Q_0[\varphi] + o(\norm{\varphi}_{\mathcal H}^2),
\end{equation*}
where
\begin{equation}\label{Q0}
\begin{aligned}
Q_0[\varphi]&=\int_0^\infty \left\lbrace (\varphi_-')^2 +\frac{1}{r^2}\varphi_-^2 + (3f^2-1)\varphi_-^2 \right\rbrace\, rdr \\
&\quad +\int_0^\infty \left\lbrace (\varphi_+')^2 +\frac{1}{r^2}\varphi_+^2 +(2f^2-1)\varphi_+^2 \right\rbrace \, rdr.
\end{aligned}
\end{equation}
Note that $Q_0[\varphi]$ is well-defined for any $\varphi\in\mathcal H$. 

\begin{lem}\label{lem_stab}
There exists $\delta>0$ such that
\begin{equation}\label{stab}
Q_0[\varphi] \geq \delta \norm{\varphi}^2_{\mathcal H},
\end{equation}
for all $\varphi\in\mathcal H$.
\end{lem} 

Part of Lemma~\ref{lem_stab}, namely the fact that $Q_0$ is non-negative, will be otained as a consequence of Mironescu's stability result \cite{mironescu95} for the classical one-component Ginzburg-Landau equation. To obtain the positive definiteness we will need an extra argument.

With Lemma~\ref{lem_stab} at hand, we will be able to construct the map $t\mapsto f^t$ as in Theorem~\ref{thm3}. The next step will be to obtain information on the sign of $f_+^t$ for $t>0$. This will be done mostly by examining the equation \eqref{bvph} satisfied by 
\begin{equation}\label{defh}
h:=\frac{d}{dt}\left[ f^t_+\right]_{t=0}.
\end{equation}
We will prove the following crucial result:
\begin{lem}\label{lem_h}
Let $h$ be a smooth function in $[0,\infty)$, satisfying the boundary value problem
\begin{gather}\label{bvph}
\Delta_r h - \frac{1}{r^2}h =(2f^2-1)h +\frac 12 f(1-f^2),\\
h(0)=0,\quad  \lim_{r\to\infty} h(r)=0.
\end{gather}
Then $h < 0$ in $(0,\infty)$. In addition, $h'(0)<0$.
\end{lem}

We now present the proofs  of Lemma~\ref{lem_stab} and Lemma~\ref{lem_h}.

\begin{proof}[Proof of Lemma~\ref{lem_stab}:]
We first remark that it suffices to establish the weaker estimate
\begin{equation}\label{weakstab}
Q_0[\varphi]\geq \delta \norm{\varphi}^2_{L^2(rdr)}\qquad\forall\varphi\in \mathcal H.
\end{equation}
Assume indeed that \eqref{weakstab} holds, but that \eqref{stab} does not. Then there is a sequence $\varphi_k$ such that $\norm{\varphi_k}_{\mathcal H}=1$ and $Q_0[\varphi_k]\to 0$. Using \eqref{weakstab}, it follows that $\norm{\varphi_k}_{L^2(rdr)}\to 0$. Since $f$ is uniformly bounded, this clearly implies that $Q_0[\varphi_k]=\norm{\varphi_k}_{\mathcal H}+o(1)$, which is absurd.

In view of the decoupled expression of $Q_0$ \eqref{Q0}, it is  enough to show that, for every $\varphi\in H^1_{loc}(0,\infty;\mathbb R)$ s.t. $\eta=\varphi(r)e^{i\theta}\in H^1(\mathbb R^2)$, it holds
\begin{equation}\label{tildeQ}
\widetilde Q [\varphi] :=\int_0^\infty \left\lbrace (\varphi')^2 +\frac{1}{r^2}\varphi^2 
+ (2f^2-1)\varphi^2\right\rbrace rdr \geq \delta \norm{\varphi}^2_{L^2(rdr)}.
\end{equation}
We appeal to Mironescu's stability result \cite{mironescu95}, which implies that, for any $\psi\in H^1(\mathbb R^2;\mathbb C)$, 
\begin{equation}\label{P}
P[ \psi ] = \int_{\mathbb R^2}\left\lbrace |\nabla\psi |^2 + (f^2-1)|\psi|^2 + 2 f^2(e^{i\theta}\cdot\psi)^2 \right\rbrace\geq 0.
\end{equation}
On the other hand, $\widetilde Q$ can be rewritten as
\begin{equation}\label{QP}
\widetilde Q [\varphi]=P[i\varphi(r)e^{i\theta}]+\int_0^\infty f^2\varphi^2\, rdr.
\end{equation}
Of course the second term in the right-hand side of \eqref{QP} is, by itself, not enough to make $\widetilde Q$ positive definite, since there exist sequences $\varphi_k$ with $\norm{\varphi_k}_{L^2}=1$ and
\begin{equation*}
\int_0^R f^2 \varphi_k^2 \, rdr \longrightarrow 0.
\end{equation*}
However, such sequences have their mass concentrated near zero, which makes the first term in the right-hand side of \eqref{QP} large. In other words, the competition between the two terms in the right-hand side of \eqref{QP} will ensure the positive definiteness of $\widetilde Q$.

Let us assume that \eqref{tildeQ} does not hold: there is a sequence $\varphi_k$ such that 
\begin{equation*}
\norm{\varphi_k}_{L^2(rdr)}=1,\quad \widetilde Q[\varphi_k]\longrightarrow 0.
\end{equation*}
Since $\widetilde Q[\varphi_k]$ is bounded, the sequence $\eta_k=\varphi_k(r)e^{i\theta}$ is bounded in $H^1(\mathbb R^2)$ and therefore weakly compact: up to extracting a subsequence, $\eta_k$ converges a.e., and strongly in $L^2_{loc}$. Hence there is a function $\varphi\in L^2(0,\infty)$ such that $\varphi_k \to \varphi$ a.e., and strongly in $L^2(0,1)$.
Since, by \eqref{QP} and \eqref{P}, 
\begin{equation*}
\int f^2\varphi_k^2 rdr \leq \widetilde Q[\varphi_k],
\end{equation*}
we deduce, using Fatou's lemma, that $\int f^2\varphi^2 rdr=0$, and therefore $\varphi\equiv 0$. In particular, it holds
\begin{equation*}
\int_0^1 \varphi_k^2 \, rdr \longrightarrow 0,
\end{equation*}
from which we infer that
\begin{align*}
\widetilde Q[\varphi_k]& \geq \int_0^\infty f^2\varphi_k^2\, rdr = \int_1^\infty f^2\varphi^2\, rdr +o(1)\\
& \geq f(1)^2 \int_1^\infty \varphi_k^2\, rdr +o(1)=f(1)^2 +o(1),
\end{align*}
contradicting the fact that $\widetilde Q[\varphi_k]\to 0$.
\end{proof}

\begin{proof}[Proof of Lemma~\ref{lem_h}:]
It is well known \cite{herve94} that $f>0$ in $(0,\infty)$. Hence we may write 
\begin{equation*}
h=fg
\end{equation*}
for some function $g$ which is smooth in $(0,\infty)$ and continuous up to $0$. In fact $g$ is smooth up to $0$, since $f(r)=r\tilde f(r)$ and $h(r)=r\tilde h(r)$ for some functions $\tilde f$ and $\tilde h$ which are smooth on $[0,\infty)$ and $\tilde f$ does not vanish on $[0,\infty)$.  

The idea of decomposing $h$ as $h=fg$ is reminiscent of Mironescu's method \cite{mironescu96} to show the radial symmetry of entire vortices of degree one in the classical one-component Ginzburg-Landau framework.

Let us compute the differential equation satisfied by g. It holds
\begin{gather*}
g'=\left( \frac hf \right)' = \frac{h'}{f}-\frac{hf'}{f^2},\\
g'' = \frac{h''}{f}-2\frac{f'h'}{f^2}-\frac{hf''}{f^2} + 2\frac{h(f')^2}{f^3}.
\end{gather*}
Therefore we find
\begin{align*}
f^2 g'' & = h''f-hf'' -2 f'h' + 2 g(f')^2 \\
& = \left[ (2f^2-1)h +\frac 12 f(1-f^2) + \frac{1}{r^2}h - \frac{1}{r}h' \right]f \\
& \quad -\left[f(f^2-1) +\frac{1}{r^2}f -\frac{1}{r}f'  \right] h -2 f'h' + 2 g(f')^2 \\
& = f^3h + \frac 12 f^2(1-f^2) -\frac 1r h'f + \frac 1r f'h -2 f'h' + 2 g (f')^2 \\
& = f^4 g + \frac 12 f^2(1-f^2) -\frac 1 r  f^2 g' -2 f' (f'g +g'f) + 2 (f')^2 g \\
& = -\left( \frac 1 r f^2 + 2 f'f \right) g'  + f^4 g + \frac 12 f^2 (1-f^2).
\end{align*}

Hence $g$ satisfies the differential equation
\begin{equation*}
g'' + \left( \frac 1r + 2 \frac{f'}{f} \right) g'  = f^2 g + \frac{1-f^2}{2},
\end{equation*}
and the boundary condition
\begin{equation*}
g(R)=0.
\end{equation*}

Recall that it holds $0<f<1$ in $(0,\infty)$. Therefore the equation implies that  $g$ can not admit a positive maximum in $(0,\infty)$, and it holds
\begin{equation*}
g\leq \max (0,g(0)).
\end{equation*}
Next we prove that $g(0) < 0$. To this end we show that $g'(0)=0$ and $g''(0)>0$. Therefore $g$ is initially increasing. In particular, if we assume that $g(0)\geq 0$, then to match the boundary condition $g(\infty)=0$, $g$ would have to attain a positive maximum inside $(0,\infty)$ which is impossible.

To show that $g'(0)=0$ and $g''(0)>0$, we perform a Taylor expansion near zero: write
\begin{equation*}
g=g_0 + g_1 r + \frac{g_2}{2}r^2 + O(r^3),\quad f=f_1 r + \frac{f_2}{2}r^2 + O(r^3),
\end{equation*}
so that
\begin{gather*}
g' = g_1 + g_2 r + O(r^2),\quad g'' = g_2 + O(r),\quad \frac{f'}{f}=\frac 1 r +\frac{f_2}{2f_1}+O(r),\\
\left(\frac 1r + 2\frac{f'}{f}\right)g' = \left(\frac 3r + \frac{f_2}{f_1}+O(r)\right)(g_1+g_2 r + O(r^2))=\frac{3g_1}{r} + 3g_2 +  g_1\frac{f_2}{f_1} + O(r) \\
g'' + \left( \frac 1r + 2 \frac{f'}{f} \right) g'  - f^2 g - \frac{1-f^2}{2}
= \frac{3g_1}{r} + 4 g_2 + g_1\frac{f_2}{f_1} -\frac{1}{2} + O(r).
\end{gather*}
Hence it holds $g_1=0$ and $g_2=1/8>0$.

As explained above, it follows that $g(0)<0$. In particular, $\max (0,g(0))=0$ and $g\leq 0$ in $[0,\infty)$. We claim that in fact this inequality is strict: it holds
\begin{equation*}
g<0 \quad\text{in } [0,\infty).
\end{equation*}
Assume indeed that $g(r_0)=0$ for some $r_0\in (0,\infty)$. Then $r_0$ is a point of maximum of $g$, so that $g''(r_0)\leq 0$. But on the other hand it holds $2g''(r_0)= 1 - f(r_0)^2 >0$, so that we obtain a contradiction. We conclude that $g<0$ in $[0,\infty)$ and therefore $h<0$ in $(0,\infty)$. Moreover, $h'(0)=f'(0)g(0)<0$.
 \end{proof}

Also of use will be the fact that the space $\mathcal H$ is embedded into the space of continuous maps vanishing at zero and infinity.
\begin{lem}\label{lem_H}
It holds 
\begin{equation*}
\mathcal H\subset \left\lbrace \varphi\in \left[C(0,\infty)\right]^2\colon \varphi(0)=0,\, \lim_{r\to\infty}\varphi(r)=0\right\rbrace,
\end{equation*}
and $\|\varphi\|_{L^\infty}\le \|\varphi\|_{\mathcal{H}}$ for all $\varphi\in\mathcal{H}$.
\end{lem}
\begin{proof}
Let $\varphi\in\mathcal H$. Then $\varphi$ is absolutely continuous in $(0,\infty)$. So are $\varphi_\pm^2$, and $(\varphi_\pm^2)'=2\varphi_\pm\varphi'_\pm$. For any $r_1\geq r_2$ it holds
\begin{align*}
\abs{\varphi_\pm(r_1)^2-\varphi_\pm(r_2)^2}&\leq 2\int_{r_1}^{r_2}\abs{\varphi_\pm}\abs{\varphi'_\pm}dr\\
& \leq \int_{r_1}^{r_2}\left[\frac{\varphi_\pm^2}{r^2}+
(\varphi'_\pm)^2\right] rdr,
\end{align*}
so that $\varphi_\pm^2$ is Cauchy at $0$ and $\infty$. Obviously the corresponding limits must be zero.  The estimate on the supremum norm follows by choosing $r_1=0$ in the inequality above.
\end{proof}

Finally, we require an asymptotic expansion of solutions which is uniform in the parameter $t$.  The following result is proven in section~\ref{s_asympt}:
\begin{thm}\label{asymptotics}
Let $[f_{t,-}, f_{t,+}]$ be solutions of \eqref{GLt}, and assume that for every $\delta>0$ there exists $R_0>0$ and $0\le T_1\le T_2\le 1$ such that for every $R>R_0$ and $t\in [T_1,T_2]$, 
\begin{equation}\label{decay}
|f_+^t(r)|\le t\delta ,\qquad |f_-^t(r)-1| \le \delta,
\end{equation}
for all $r\ge R$.
Then we have 
\begin{equation}  \label{asym}
f_{t,-}= 1 - \frac{1}{2 r^{2}} -\frac{5t^2+9}{8 r^{4}} + O(r^{-6}), \qquad
f_{t,+} = t\left[-{1\over 2 r^2} -{13\over 4r^4} + O(r^{-6})\right],
\end{equation}
as $r\to\infty$.
More precisely, there exist positive constants $C_\pm, C'_\pm, R>0$ such that
\begin{gather}\label{fmasy}
\left|f_{t,-}(r)-\left(1 - \frac{1}{2 r^{2}} -\frac{5t^2+9}{8 r^{4}}\right)\right|\le\frac{C_-}{r^6} ,  \\
\label{fpasy}
\left|f_{t,+}(r)+t\left[{1\over 2 r^2} +{13\over 4r^4}\right]\right|\le t\frac{C_+}{r^6}\\
\label{f'masy}
\left|f'_{t,-}(r)+\frac{1}{r^3}\right|\le\frac{C'_{-}}{r^5} ,\\
\label{f'pasy}
\left|f'_{t,+}(r)+\frac{t}{r^3}\right|\le t\frac{C'_{+}}{r^5} ,
\end{gather}
hold for all $r\ge R$ and all $t\in [T_1,T_2]$.
\end{thm}

\subsection{Proof of Theorem~\ref{thm3}}\label{ss_proof}

\textbf{Step 1:} Construction of the family $t\mapsto f^t$.

We denote by $\mathcal N_t(f)$ the quasilinear differential operator such that 
\begin{equation*}
\langle DI_t(f;R),\varphi\rangle_{\mathcal (H^0_R)^*,\mathcal H^0_R} = \langle\mathcal N_t(f),\varphi\rangle_{L^2(rdr)}.
\end{equation*}
for $\varphi\in C_c^\infty(0,R)$. In other words, the system \eqref{GLt} is exactly $\mathcal N_t(f)=0$.

Using the fact that $\mathcal N_0 (f^0)=0$, one may check that
\begin{equation*}
\mathcal N_t(f^0+g)\in \mathcal H^*\qquad\forall g\in\mathcal H.
\end{equation*}
Moreover, the map
\begin{equation*}
\mathcal F \colon (-1,1)\times \mathcal H\to\mathcal H^*,\quad (t,g) \mapsto \mathcal N_t(f^0+g),
\end{equation*}
is smooth.
Since
\begin{equation*}
\langle D_g\mathcal F(0,0)\varphi,\varphi\rangle_{\mathcal H^*,\mathcal H}=Q_0[\varphi],
\end{equation*}
Lemma~\ref{lem_stab} and Lax-Milgram theorem imply that $D_g\mathcal F(0,f^0)$ is invertible. Applying the implicit function theorem, we find that there exists $t_0>0$, $\delta_0>0$ and a smooth map
\begin{equation*}
(-t_0,t_0)\ni t\mapsto g^t \in\mathcal H,\quad g^0=0,
\end{equation*}
such that, for  $\abs{t}<t_0$ and $\norm{g}_{\mathcal H}<\delta_0$,
\begin{equation}\label{implicit}
\mathcal F (t,g)=0 \quad\Longleftrightarrow\quad g=g^t.
\end{equation}
In particular, $f^t=f^0+g^t$ solves \eqref{GLt}. Elliptic regularity ensures that for every $t$, $f^t$ is a smooth function. 


\textbf{Step 2:} The map $t\mapsto f^t \in C^k([0,R])$ is smooth, for any integer $k$ and $R>0$.

In fact we consider spaces of differentiable functions which are more appropriate to our problem: let
\begin{equation*}
\widetilde C^k(0,R)=\left\lbrace f_\pm\in C^k(0,R)\colon \eta_\pm=f_\pm(r)e^{\pm i\theta}\in C^k(\overline B_R)\right\rbrace.
\end{equation*}

Let $t_1\in (-t_0,t_0)$. Since (by Lemma~\ref{lem_H}) $\mathcal H$ is embedded in a space of continuous functions, the map $t\mapsto g^t(R)$ is smooth, and we may fix a smooth map $t\mapsto \psi^t\in \widetilde C^{k+2}(0,R)$ such that
\begin{equation*}
\psi^0\equiv 0,\quad (\psi^t + g^{t_1})(R)=g^{t_1+t}(R).
\end{equation*}
Next we consider the smooth map
\begin{equation*}
\widetilde{\mathcal F}\colon (-\varepsilon,\varepsilon)\times \widetilde C^{k+2}(0,R)\to \widetilde C^k(0,R),\quad (t,g)\mapsto \mathcal N_{t_1+t}(f^{t_1}+\psi^t +g).
\end{equation*}
The small constants $t_0$ and $\delta_0$ in Step~1 may be chosen so that
\begin{equation*}
\langle D_g\mathcal F(t,g)\varphi,\varphi\rangle_{\mathcal H^*,\mathcal H}\geq c\norm{\varphi}^2_{\mathcal H},\qquad \abs{t}<t_0,\;\norm{g}<\delta_0,
\end{equation*}
for some $c>0$. It is then easy to check, using elliptic regularity, that $D_g\widetilde{\mathcal F}(0,0)$ is invertible. Therefore the implicit function theorem provides us with a smooth family $t\mapsto\tilde g^t\in \widetilde C^{k+2}(0,R)$ defined for small $t$ and solving
\begin{equation*}
\mathcal N_{t_1+t}(f^{t_1}+\psi^t +\tilde g^t)=0.
\end{equation*}
For small enough $t$, the function
\begin{equation*}
\hat g^t =\begin{cases}
 g^{t_1}+\psi^{t}+\tilde g^{t} & \text{ in }(0,R),\\
 g^{t_1+t} & \text{ in }(R,\infty),
 \end{cases}
\end{equation*}
satisfies $\norm{\hat g^t}_{\mathcal H}<\delta_0$. Moreover, it holds $\mathcal F(t_1+t,\hat g^t)=0$, so that by \eqref{implicit} we deduce that $\hat g^t =g^{t_1+t}$. In particular, the  map $t\mapsto g^t\in \widetilde C^{k+2}(0,R)$ is smooth.

\textbf{Step 3:} It holds $f_+^t< 0$ and $0< f^t_-< 1$ in $(0,\infty)$ for small enough $t$.

Let $\phi^t ={\partial\over\partial t} f^t$.  By Step 2, the map $t\mapsto \phi^t$ is smooth in $\mathcal{H}\cap C^k_{loc}$ for each $k$, and hence $\phi^t$ solves the system obtained by differentiating the equations \eqref{systftdiag} with respect to $t$.  As $\phi^t$ is continuous at $t=0$, a computation reveals that $\phi^0=(\phi^0_-,\phi^0_+)\in \mathcal{H}\cap C^k_{loc}$, with $\phi^0_+=h$, the solution of \eqref{bvph} and $\phi^0_-$ solving the linearized radial Ginzburg-Landau equation, 
$\Delta_r \phi_-^0 -\frac{1}{r^2}\phi_-^0 = \phi_-^0(3f^2-1)$, and thus, $\phi^0_-=0$.  As the map $t\mapsto f^t$ is smooth in $\mathcal{H}\cap C^k_{loc}$, it follows that $f^t = f^0 + t \phi^0 + O(t^2)$, with error term uniform in supremum norm on $[0,\infty)$, by Lemma~\ref{lem_H}.  Since $\phi_\pm^0(r)\to 0$ as $r\to\infty$, for any $\delta>0$ we may find $R_0>0$ such that $|\phi_\pm^0(r)|<{\delta\over 2}$ for all $r\ge R_0$.  By the Taylor expansion of $f^t$ we may then conclude that for any $R\ge R_0$, there exists $T>0$ for which
 $$|f_+^t(r)|\le t\delta, \ \text{and} \ |f_-^t(r)-1|\le\delta, $$
  for all $r\ge R$ and $t\in [0,T]$.  

The solutions $f^t$ thus satisfy the hypotheses of  Theorem~\ref{asymptotics}, therefore we may choose $R>0$ such that for all $t\in (0,T]$,
\begin{equation*}
f_+^t<0\quad\text{and}\quad 0<f_-^t<1\qquad\text{in }[R,\infty).
\end{equation*}
Thus it only remains to show that $f_+^t< 0$ and $0< f^t_-< 1$ in $(0,R)$ for small enough $t$.

It is well-known \cite{herve94} that $(f_-^0)'=f'\geq c>0$ in $(0,R)$, so that Step 2 ensures that $(f_-^t)'>0$ in $(0,R)$ for small enough $t$, and we deduce that $0<f_-^t<1$ in $(0,R)$.

Next we show that $f_+^t<0$.  We recall that $h =\frac{\partial}{\partial t}\left[f_+^t\right]_{t=0}$ solves \eqref{bvph}.
In view of Lemma~\ref{lem_H}, $h$ is bounded and satisfies $h(0)=h(\infty)=0$. Elliptic regularity ensures that $h$ is smooth in $[0,\infty)$, and we may apply Lemma~\ref{lem_h}. Thus it holds $h<0$ in $(0,\infty)$, and $h'(0)<0$. There exists $r_0>0$ and $\eta>0$ such that 
\begin{equation*}
h'(r)\leq -\eta\;\text{ in }[0,r_0],\qquad h(r)\leq -\eta\;\text{ in }[r_0,R].
\end{equation*}
Using Step~2, we infer that for all small enough $t$, 
\begin{equation*}\frac{\partial}{\partial t}[(f_+^t)']\leq -\eta/2<0\;\text{ in }[0,r_0],\qquad \frac{\partial}{\partial t}[ f_+^t]\leq -\eta/2\;\text{ in }[r_0,R],
\end{equation*}
which obviously implies, since $f_+^0\equiv 0$, that $f_+^t<0$ in $(0,R]$.
\qed

\section{Asymptotics}\label{s_asympt}

We derive the asymptotic behavior of solutions $f_\pm(r)$ as $r\to\infty$ by means of the sub- and super-solutions method. 
We recall the notation for the Laplacian of radial functions in $\RR^2$, 
$\Delta_r u(r):= \frac1r (r\, u'(r))'$.
This we accomplish thanks to the following comparison lemma, which is an adaptation of Lemma~3.1 in \cite{alamagao13}:

\begin{lem}\label{comparison}
Let $\mathbb{A, B, C, D}$ be bounded functions on $[R,\infty)$, with $\mathbb{A},\mathbb{D}>0$, $\mathbb{B},\mathbb{C}\le 0$, and such that the quadratic form defined by $\mathbb{A, B, C, D}$ satisfies the bound
\begin{equation}\label{comp0}  \mathbb{A}(r)x^2 + (\mathbb{B}(r)+\mathbb{C}(r))xy + \mathbb{D}(r)y^2 
   \ge \delta(x^2+y^2), 
\end{equation}
for all $r\in [R,\infty)$, $(x,y)\in\RR^2$, and constant $\delta>{1\over 2R^2}.$
 Then, if $u, v$ satisfy:
$$
\left\{
\begin{gathered}
 -\Delta_r u+ {1\over r^2}u+\mathbb{A}u+\mathbb{B} v\le 0,\\
 -\Delta_r v+{1\over r^2}v+\mathbb{C}u+\mathbb{D} v\le 0,
\end{gathered}
\right.
$$
for $r\in (R,\infty)$, with 
$$
u(R)\le0, \quad v(R)\le0, \qquad u(r), v(r)\to 1 \quad \text{as $r\to\infty$.}
$$  
  we have that $u\le0$ and $v\le0$ in $[R, \infty)$.
\end{lem}

\begin{proof}
Let $u^{\pm}=\text{max} (\pm u, 0)$ and $v^{\pm}=\text{max} (\pm v, 0)$, the positive and negative parts of each component.
We set $\eta_R(r)=e^{-(r-R)/R}$, $r\in [R,\infty)$, multiply the first equation by $u^{+}\eta_R$ and the second equation by $v^{+}\eta_R$, integrate over $[R, \infty)$, and add the two resulting inequalities, to obtain:
\begin{multline}\label{comp1}  \int_R^\infty \biggl\{ 
-u^+\Delta_r u -v^+\Delta_r v  + {(u^+)^2 + (v^+)^2\over r^2} + \\
    \mathbb{A}(u^+)^2 + \mathbb{B} u^+v + \mathbb{C} v^+u + \mathbb{D} (v^+)^2
\biggr\}\eta_R\, r\, dr\le 0.
\end{multline}
Applying \eqref{comp0}, all but the first two terms  in \eqref{comp1} may be bounded as follows:
\begin{align*} &
\int_R^\infty \left[ {(u^+)^2 + (v^+)^2\over r^2} +
    \mathbb{A}(u^+)^2 + \mathbb{B} u^+v + \mathbb{C} v^+u + \mathbb{D} (v^+)^2
\right]\eta_R\, r\, dr  \\
&\qquad \ge
\int_R^\infty \left[ 
    \mathbb{A}(u^+)^2 + \mathbb{B} (u^+v^+ - u^+v^-) 
       + \mathbb{C} (v^+u^+ - v^+ u^-) + \mathbb{D} (v^+)^2
\right]\eta_R\, r\, dr  \\
&\qquad \ge \int_R^\infty \left[ 
    \mathbb{A}(u^+)^2 + (\mathbb{B}  + \mathbb{C}) v^+u^+ + \mathbb{D} (v^+)^2
\right]\eta_R\, r\, dr  \\
&\qquad\ge \delta \int_R^\infty \left[ (u^+)^2 + (v^+)^2\right]\eta_R\, r\, dr.
\end{align*}

Integrating the first term by parts, using the hypothesis $u(R)\le 0$ and the explicit form of $\eta_R$,
we obtain:
\begin{align*}
-\int_R^\infty \eta_R\,u^+\Delta_r u\, r\, dr
&= u^+(R)\eta_R(R)Ru'(R) + \int_R^\infty \left\{ \eta_R [(u^+)']^2 + {1\over R}\eta_R u^+\, [u^+]'\right\} r\, dr \\
&\ge \frac12\int_R^\infty \eta_R [(u^+)']^2\, r\, dr 
- {1\over 2R^2} \int_R^\infty \eta_R \, [u^+]^2\, r\, dr.
\end{align*}
An analogous computation may be made for the second term (involving $v^+$), and inserting in \eqref{comp1} we conclude that
$$ 0\ge \int_R^\infty \left\{ ([u^+]')^2 + ([v^+]')^2 
     + \left(\delta-{1\over 2R^2}\right)([u^+]^2 + [v^+]^2)\right\}\eta_R\, r\, dr.  $$
As $\eta_R>0$ on $[R,\infty)$, we conclude that $u^+, v^+\equiv 0$ on $[R,\infty)$, and the lemma is proven.
\end{proof}

The proof of the asymptotic formulae in Theorem~\ref{asymptotics} thus relies on the construction of appropriate sub- and super-solutions for the system \eqref{GLt}. By taking linear combinations of the equations \eqref{GLt} we can rewrite the system in `diagonalized' form:
\begin{equation}\label{systftdiag}
\begin{gathered}
(1-\frac{t^2}{4})\left( \Delta_r f_- -\frac{1}{r^2}f_-\right) = f_-(2f_+^2+f_-^2-1) -\frac t2 f_+ (2f_-^2+f_+^2-1), \\
(1-\frac{t^2}{4})\left( \Delta_r f_+ -\frac{1}{r^2}f_+\right) =f_+(2f_-^2+f_+^2-1)-\frac t2 f_-(2f_+^2+f_-^2-1),
\end{gathered}
\end{equation}
which will be more convenient to work with during the proof of Theorem~\ref{asymptotics}.

\begin{proof}[Proof of Theorem~\ref{asymptotics}] For simplicity of notation, we denote $f^t_{\pm}=f_\pm$ in the proof, suppressing the dependence on $t$. 
We also denote by $\tau:=1-{t^2\over 4}\in [\frac34,1)$.

\medskip

\noindent
{\bf Step 1:} Construction of subsolution/supersolution pairs.  We begin with supersolutions. Let
\begin{align}
 w_{+}=t\left[{a_{+}\over r^2}+{b_{+}\over r^4}+c_{+}\frac{R^6}{r^6}\right],\ \label{wp}\\
 w_{-}=1+{a_{-}\over r^2}+{b_{-}\over r^4}+c_{-}\frac{R^6}{r^6},\ \label{wm}
\end{align}
where $a_{\pm}, b_{\pm}, c_\pm$ and $R$ are to be chosen so that
\begin{gather} \label{E-}
E_-:= [-\tau\Delta_R w_- +{w_-\over r^2}] + w_-(2w_+^2+w_-^2-1)
                 -{t\over 2}w_+(2w_-^2+w_+^2-1) \ge 0, \\
\label{E+}
E_+:= [-\tau\Delta_R w_+ +{w_+\over r^2}] + w_+(2w_-^2+w_+^2-1)
                 -{t\over 2}w_-(2w_+^2+w_-^2-1) \ge 0,
\end{gather}
for all $r\ge R$, and 
\begin{equation}\label{BC}
w_-(R)\ge f_{-}(R), \qquad w_+(R)\ge f_{+}(R).
\end{equation}

Expanding \eqref{E+} and \eqref{E-} yields terms which are polynomials in even powers of $r^{-1}$, of the form:
\begin{equation*}
E_+= t\sum^{9}_{k=1} M^{+}_{2k} {1\over r^{2k}}, \qquad
E_-= \sum^{9}_{k=1} M^{-}_{2k}{1\over r^{2k}},
\end{equation*}
where $M^{\pm}_{2k}=M^{\pm}_{2k}(t, R, a_{\pm}, b_{\pm}, c_\pm)$ is a polynomial in each of its arguments.  The expansion is quite horrific, but may be explicitly evaluated with the help of a symbolic algebra program such as Maple.
First, we choose $a_\pm$ in order to force the lowest order coefficients
 $M_2^\pm$ to vanish:  indeed, the expansion yields 
\begin{equation*}
M^{-}_{2}=2a_- - {t^2\over 2}a_+ + \tau=0, \qquad 
  M^+_2=-a_-+a_+=0,
\end{equation*}
which gives the coefficients of $r^{-2}$, $a_-=-\frac12=a_+$, as in \eqref{asym}.

Similarly, we fix the values of the coefficients $b_\pm$ in order that the $r^{-4}$ terms vanish,
\begin{gather*}
M^{-}_{4}=2b_- -{t^2\over 2}b_+ -3\tau a_- + 3a_-^2 -  2t^2a_+a_- +2t^2 a_+^2=0, 
\\
M^+_4 = b_+  -b_- -3\tau a_+ -\frac32 a_-^2 -t^2 a_+^2 + 4 a_+a_- =0.
\end{gather*}
Thus, $b_-= -{5t^2 +9\over 8}$, $b_+= -{13\over 4}$ are the coefficients of $r^{-4}$ given in the expansion \eqref{asym}.

The values of $a_{\pm}, b_{\pm}$ may then be substituted into the expansions of \eqref{E+} and \eqref{E-}, and the expressions for $M^\pm_{2k}$ may be viewed as functions of $R$. The exact form of the coefficients $M^\pm_{2k}$ is very complex, but they are all polynomials in $R$, $t$, and $c_\pm$. As we will choose $R$ large, we are only interested in the leading order of each. We obtain:
$$
\begin{gathered}
M^{+}_{6}= (-c_-+c_+)R^6 + O(1), 
  \qquad M^-_6 = \left(2c_- -{t^2\over 2}c_+\right) R^6+ O(1), \\
M^{\pm}_{8}=O(R^{6}),\qquad M^{\pm}_{10}=O(R^{6}),\\
M^{+}_{12}=\left( 4c_+ c_- -{t^2\over 2}(2c_+^2+3c_-^2)\right)R^{12} + O(R^6), \\
M^{-}_{12}=(-2t^2c_+c_- +2t^2c_+^2 +3 c_-^2)R^{12}+O(R^{6}),\\
M^{\pm}_{14}=O(R^{12}),\qquad M^{\pm}_{16}=O(R^{12}),\\
M^{+}_{18}=\left(c_+^3t^2 -c_-c_+^2t^2 + 2c_+c_-^2-\frac12 c_-^3\right)R^{18},\\
M^{-}_{18}=\left(c_-^3 + 2c_+^2c_- t^2 -c_+c_-^2 t^2 -\frac12 c_+^3 t^4\right)R^{18}.
\end{gathered}
$$
In each expression, the lower terms are uniformly bounded for $t\in [0,1]$.

Let $c_-=\delta$ and $c_+=2\delta$, with $\delta>0$ to be chosen later.  With this definition,
$$   M^+_6 = \delta R^6 + O(1), \qquad M^-_6 = \delta(2-t^2)R^6 + O(1), $$
where the remainder terms are uniformly bounded for $t\in [0,1]$.  
As $M^\pm_{6}$ are the leading order terms in $r$, this will ensure that we obtain the correct sign in each equation, and the value of $\delta$ will be fixed in order that the $r^{-6}$ terms indeed dominate the others in the expansion.
By choosing $R_1=R_1(\delta)$ sufficiently large, we may then ensure that when $R\ge R_1$,
\begin{equation}\label{4sumbound}
\left|{M^\pm_{8}\over r^8}+{M^\pm_{10}\over r^{10}}
  +{M^\pm_{14}\over r^{14}}
  +{M^\pm_{16}\over r^{16}}\right| \le C{R^6\over r^8} < {\delta\over 4}{R^6\over r^6},
\end{equation}
for all $r\in [R,\infty)$, with constant $C$ chosen independent of $t\in [0,1]$.
Next, with our choice of $c_\pm$, we have 
$$  |M^\pm_{12}|\le 7\delta^2 R^{12} + O(R^6)\le 8\delta^2 R^{12},  $$
for all $R\ge R_1$, making $R_1$ larger if necessary.
Hence we may fix $\delta$ with $0<\delta<{1\over 32}$, we have:
$$  \left|{M^\pm_{12}\over r^{12}}\right| \le {8\delta^2 R^{12}\over r^{12}} < {\delta\over 4}{R^6\over r^6}, $$
holds for all $r\in [R,\infty)$ with $R\ge R_1$.
Finally, we note that 
$$  M^+_{18}= \left(\frac72 + 4t^2\right)\delta^3, \qquad M^-_{18}=(1+6t^2-4t^4)\delta^3, $$
and for $t\in [0,1]$ each has the same sign as $\delta$, and thus these terms contribute with the desired sign in the evaluation of \eqref{E+}, \eqref{E-}, and may be neglected.

Putting these estimates together, it follows that for all $R\ge R_1$, 
$$  E_\pm\ge M_6^\pm r^{-6} - \left|\sum_{k=4}^8 M_{2k}^+ r^{-2k}\right| 
 \ge M_6^\pm r^{-6} - {\delta\over 2} {R^6\over r^6}> {\delta\over 4}{R^6\over r^6}>0, 
$$
for all $r\in [R,\infty)$, and uniformly in $t\in (0,1]$.  Thus, $(w_-, w_+)$ indeed satisfy the supersolution conditions \eqref{E-} and \eqref{E+} for $R\ge R_1$, as desired.

It remains to consider the behavior at the endpoint, $r=R$.  Since
$$   w_-(R) = 1 + \delta + O(R^{-2}), \quad  w_+(r)= 2t\delta + O(R^{-2}), $$
with $0<\delta<\frac{1}{32}$, by the hypothesis \eqref{decay} we may fix $R\ge R_1$  such that  $f_-(R)\le w_-(R)$ and $f_+(R)\le w_+(R)$ holds for all $t\in [T_1,T_2]$.  Thus \eqref{BC} holds as well, and we have completed the construction of supersolutions.

We also require a subsolution pair, $(z_-,z_+)$ for which $E_+\le 0$ and $E_-\le 0$ for all $r\in [R,\infty)$ and $z_-(R)\le f_-(R)$, $z_+(R)\le f_+(R)$, for $R$ sufficiently large proceeds in exactly the same way as for the supersolution pair above, except the coefficients 
$\frac12 c_+=c_-=-\delta<0$.  This completes Step 1 in the proof.

{\bf Step 2:} We apply the comparison Lemma~\ref{comparison} to the pair 
$(h_-,h_+)=(f_- - w_-, f_+ - w_+)$.  Denote by $Lu:= -\Delta_r u + r^{-2}u$  Then, an explicit calculation together with the construction of Step 1 shows that, for any sufficiently large $R$,
\begin{equation}\label{subsol}  \left\{ \begin{gathered}
Lh_- +\mathbb{A} h_- + \mathbb{B} h_+ \le 0\\
Lh_+ +\mathbb{C} h_- + \mathbb{D} h_+\le 0 \\
\end{gathered}\right.
\end{equation}
for $r\in [R,\infty)$, with $h_\pm(R)\le 0$.  The coefficients are functions of $r$, but have uniform limits as $r\to\infty$,
\begin{align*}
\mathbb{A}&= f_-^2 + f_- w_- + w_-^2 + 2f_+^2 -1-{t\over 2}(2 w_+ (f_- + w_-)) \longrightarrow 2, \\
\mathbb{B}&=2w_- (f_+ + w_+)-{t\over 2}(f_+^2 + f_+w_+ + w_+^2 + 2f_-^2 -1) \longrightarrow -{t\over 2},\\
\mathbb{C}&= 2 w_+ (f_- + w_-)-{t\over 2}(f_-^2 + f_- w_- + w_-^2 + 2f_+^2 -1) \longrightarrow -t,\\
\mathbb{D}&= f_+^2 + f_+w_+ + w_+^2 + 2f_-^2 -1-{t\over 2}(2w_- (f_+ + w_+)) \longrightarrow 1.
\end{align*}
Thus (taking $R$ larger if necessary) we may assume
the positivity condition \eqref{comp0} is satisfied in $[R,\infty)$ with $\delta= \frac34$, for example.  Lemma~\ref{comparison} applies, and we conclude that $h_\pm(r)\le 0$ on $[R,\infty)$, that is $f_\pm(r)\le w_\pm(r)$.

Taking $(h_-,h_+)=(z_- - f_-, z_+ - f_+)$, with $(z_-, z_+)$ the subsolution pair, we may repeat the above computations to arrive at the same system \eqref{subsol}, with coefficients 
$\mathbb{A,B,C,D}$ satisfying the same asymptotic conditions as above.  Thus, by Lemma~\ref{comparison} we may also conclude that $f_\pm(r)\ge z_\pm(r)$ for all $r\in [R,\infty)$, for $R$ sufficiently large.  Using the explicit form of $z_\pm, w_\pm$ from Step 1, we may conclude that the estimates \eqref{fmasy} and \eqref{fpasy} both hold.

\medskip

{\bf Step 3:} The derivative estimates.  Here we follow the method of Chen, Elliot, and Qi \cite{CEQ}.  Let 
\begin{align*}  g_-(r)&= f_-(r) - \left( 1 + {a_-\over r^2} + {b_-\over r^4}\right)
                   = f_-(r)  - w_-(r) + {c_-\over r^6}, \\
                   g_+(r)&= f_+(r) - \left( {a_+\over r^2} + {b_+\over r^4}\right)
                   = f_+(r)  - w_+(r) + {c_+\over r^6},
\end{align*}
where $a_\pm, b_\pm, c_\pm$ are as in Step 1.  By Step 2, we thus know that $g_\pm(r)=O(r^{-6})$.  A calculation then yields
\begin{gather*}  \Delta_r g_- = {g_-\over r^2}+\mathbb{A} g_- + \mathbb{B} g_+ +O(r^{-6}) = O(r^{-6}),
\\
\Delta_r g_+= {g_+\over r^2}+\mathbb{C} g_- + \mathbb{D} g_+ +O(r^{-6}) = O(r^{-6}),
\end{gather*}
uniformly for $t\in [T_1,T_2]$, with $\mathbb{A,B,C,D}$ as in Step 2.

For each $k\in\RR$ there exists $r_k\in (k, 2k)$ such that 
$$g'_\pm(r_k) = {g_\pm(2k)-g_\pm(k)\over k}=O(k^{-7})=O(r_k^{-7}). $$
Integrating the estimate on $\Delta_r g_\pm$, we have, for all $r\ge R$, 
$$
|rg'_{\pm}(r)-r_{k}g'_{\pm}(r_{k})|=\left|\int^{r_{k}}_{r}\Delta_r g_\pm(r) \,r\, dr\right|
\le C\int^{r_{k}}_{r}r^{-5}dr\le \frac{4C}{r^4},
$$  
with constant $C>0$.  We now let $k\to\infty$, and use $r_k g'_\pm(r_k)\to 0$, to obtain
$|rg'_\pm(r)|\le {4C\over r^4}$, and hence $|f'_\pm(r) + {2a_\pm\over r^3}|\le {C'\over r^5}$, which gives \eqref{f'masy}, \eqref{f'pasy}.
\end{proof}

\bibliographystyle{plain}
\bibliography{pwave}

\end{document}